\documentclass[reqno]{amsart}
\usepackage{amssymb,amsmath,amsthm,amstext,amsfonts}
\usepackage{amsmath,amstext,amsthm,amsfonts}
\usepackage[dvips]{graphicx}
\usepackage{epic,eepic}
\usepackage[dvips]{graphicx}
\usepackage{psfrag}
\usepackage{color}

\makeatletter \@addtoreset{equation}{section} \makeatother

\renewcommand\thetable{\thesection.\@arabic\c@table}

\theoremstyle{plain}
\newtheorem{maintheorem}{Theorem}

\newtheorem{maincorollary}{Corollary}
\newtheorem{theorem}{Theorem }[section]

\newtheorem{lemma}[theorem]{Lemma}

\theoremstyle{definition} \theoremstyle{remark}
\newtheorem{remark}[theorem]{Remark}
\newtheorem{example}[theorem]{Example}
\newtheorem{definition}[theorem]{Definition}

\def \a{\alpha}

\def \d{\delta}
\def \e{\varepsilon}
\def \eps{\varepsilon}

\def \R{\mathbb{R}}

\newcommand{\si}{\sigma}
\newcommand{\Si}{\Sigma}
\newcommand{\cM}{\mathcal M}
\newcommand{\de}{\delta}
\newcommand{\cQ}{\mathcal Q}

\newcommand{\cE}{\mathcal{E}}

\begin{document}

\title{Entropy formulas for dynamical systems with mistakes}
\author{J\'er\^ome Rousseau and  Paulo Varandas and Yun Zhao}

\address{J\'er\^ome Rousseau, Universit\'e Europ\'eenne de Bretagne, Universit\'e de Brest, Laboratoire de Math\'ematiques CNRS UMR 6205, 6 avenue Victor le Gorgeu, CS93837, F-29238 Brest Cedex 3, France}
\email{jerome.rousseau@univ-brest.fr}
\urladdr{http://pageperso.univ-brest.fr/$\sim$rousseau}

\address{Paulo Varandas, Departamento de Matem\'atica, Universidade Federal da Bahia\\
Av. Ademar de Barros s/n, 40170-110 Salvador, Brazil.}
\email{paulo.varandas@ufba.br}
\urladdr{www.pgmat.ufba.br/varandas}

\address{Yun Zhao, Department of Mathematics, Soochow University
Suzhou 215006, P.R. China}
\email{zhaoyun@suda.edu.cn}

\footnotetext{ Research supported in part by NSFC(11001191 ), by
NSF in Jiangsu Province(09KJB110007), by a Pre-research Project of
Soochow University, by CNPq-Brazil, by FAPESB-Brazil and by
DynEurBraz.}

\date{\today}

\maketitle

\begin{abstract}
We study the recurrence to
mistake dynamical balls, that is, dynamical balls that admit some
errors and whose proportion of errors decrease tends to zero with
the length of the dynamical ball. We prove, under mild assumptions, that the
measure-theoretic entropy coincides with the exponential growth
rate of return times to mistake dynamical balls and that minimal
return times to mistake dynamical balls grow linearly with respect
to its length. Moreover we obtain averaged recurrence formula for
subshifts of finite type and suspension semiflows. Applications include
$\beta$-transformations, Axiom A flows and suspension semiflows
of maps with a mild specification property. In particular we extend some results from
\cite{ch,msuz,var} for mistake dynamical balls.
\end{abstract}

\section{Introduction.}
\setcounter{section}{1}
\setcounter{equation}{0}
 Throughout this paper, $(X,f)$ denotes a topological
 dynamical systems (TDS for short) in the sense that $f:X\rightarrow X$ is a continuous transformation
 on the compact metric space $X$ with the metric $d $. Invariant Borel probability measures are
associated with $(X,f)$. The terms $\mathcal{M}(X,f)$ and
$\mathcal{E}(X,f)$  represent  the space of $f$-invariant Borel
probability measures and the set of $f$-invariant ergodic Borel
probability measures, respectively.

The well known notions of topological and measure-theoretic entropy
constitute important invariants in the characterization of the complexity of a dynamical system.
Just as an illustration let us mention that the measure-theoretic entropy
 turned out to be a surprisingly universal concept in ergodic
 theory since it appears in the study of different subjects as
 information theory. We refer the reader to \cite{kat} for a rather complete overview.

An important characteristic of invariant measures is recurrence. Poincar\'e recurrence theorem is one of the basic
but fundamental results of the theory of dynamical  systems and it essentially states that each dynamical  system
 preserving a finite invariant measure exhibits a non-trivial recurrence to each set with positive measure. More precisely, it asserts that
if $A\subset X$ is a measurable subset of positive $\mu$-measure,
then $\hbox{Card}\{n: f^nx\in A\}=\infty$ for $\mu$-almost every
point $x\in A$.

Given a dynamical system $(X,f)$, a natural question is: which kind of further information can be obtained
when the subset $A$ is replaced by a decreasing sequence of sets $U_n$?
 There are some interesting results for this question. Ornstein and Weiss \cite{ow}
 proved that the entropy $h_{\mu}(f,\mathcal{Q})$ of an ergodic
 measure $\mu$ with respect to a partition $\mathcal{Q}$ is given
 by the almost everywhere well-defined limit
 \[
 h_{\mu}(f, \mathcal{Q})=\lim_{n\rightarrow\infty} \frac{1}{n}
 \log R_n(x,\mathcal{Q})
 \]
where $R_n(x,\mathcal{Q})=\inf \{ k\geq 1: f^k(x)\in
\mathcal{Q}^{n}(x)\}$ is the $nth~return~time$ with respect to the
partition $\mathcal{Q}$,
$\mathcal{Q}^n=\bigvee_{i=0}^{n-1}f^{-i}\mathcal{Q}$ is the
refined partition and $\mathcal{Q}^{n}(x)$ denotes the element of
$\mathcal{Q}^n$ which contains the point $x$. In consequence, the
measure-theoretic entropy is the supremum of the exponential
growth rates of Poincar\'e recurrences over all finite measurable
partitions. Downarowicz and  Weiss \cite{dw} proved that the
measure-theoretic entropy is given by the exponential growth rate
of return times to dynamical balls. More recently,  in the study of the relation
betweed entropy, dimension and Lyapunov exponents, the second author
in~\cite{var} used combinatorial arguments to provide an
alternative and more direct proof of this result. Namely, when
$\mu$ is an $f$-invariant ergodic measure, the measure theoretical
entropy can be computed for $\mu$-almost every $x\in X$ by the
following limits:
\[
h_{\mu}(f)=\lim_{\e\rightarrow
0}\limsup_{n\rightarrow\infty}\frac{1}{n} \log
R_n(x,\e)=\lim_{\e\rightarrow
0}\liminf_{n\rightarrow\infty}\frac{1}{n} \log R_n(x,\e),
\]
where $R_n(x,\e)=\inf \{ k\geq 1: f^k(x)\in B_n(x,\e)\}$ is the first return time to the dynamical ball
$B_n(x,\e)=\{ y\in X: d(f^i(x),f^i(y))<\e,~0\leq i\leq n-1\}$. In particular, return times to dynamical
balls grow exponentially fast with respect to every ergodic measure with positive entropy.
The same result is no longer true for minimal return times. Indeed,
when $f$ has the specification property, in \cite{var}
the second author obtained also that the minimal return times to dynamical balls defined by
$S_n(x,\e)=\inf \{ k\geq 1: f^{-k}(B_n(x,\e))\cap B_n(x,\e)\neq \emptyset\}$
grow linearly with $n$, that is,
\[
1=\lim_{\e\rightarrow
0}\limsup_{n\rightarrow\infty}\frac{1}{n}S_n(x,\e)=\lim_{\e\rightarrow
0}\liminf_{n\rightarrow\infty}\frac{1}{n}S_n(x,\e)\qquad\textrm{for $\mu$-a.e. $x\in
X$},
\]
for any $\mu\in \mathcal{E}(X,f)$ satisfying $h_{\mu}(f)>0$.

More recently, Marie and Rousseau~\cite{mr} initiated the study of recurrence
properties for random dynamical systems. The authors established relations between
random recurrence rates and local dimensions of the stationary measure of
the random dynamical systems under some natural assumptions. To
the best of our knowledge, this is the first step in the study of
recurrence behavior in random dynamical systems. We report some progress to obtain
Ornstein-Weiss type of formulas in the random setting in~\cite{rvz}.

An important contribution was given also by Maume-Deschamps, Schmitt, Urbanski
and Zdunik in \cite{msuz}, where the authors explored the connection between recurrence
and topological pressure of any H\"older continuous potential for subshifts of
finite type. In fact, the authors studied  return time with some
weighted function, and they obtained some interesting relations
between pressure and return times. Related results were obtained
by Meson and Vericat in \cite{mv} for the more general setting of homeomorphisms with
the specification property.

Here we will refer to return times to mistake dynamical balls, whose precise definition
will be given later on.
Roughly, when a physical process evolves it is natural that it may change or that
some errors are committed in the evaluation of orbits. However, if
the system is self adaptable the amount proportion of errors
should decrease as the time evolves. This gives us the motivation
to consider return times to mistake dynamical balls, whose
formalization is also in connection with the almost specification
property introduced by Pfister and Sullivan~\cite{ps} and  Thompon~\cite{th}.
We refer the reader to the beginning of the next section for the precise definition
of mistake dynamical balls.
Since the proportion of admissible errors
decreases as the time evolves we can prove that the
measure-theoretic entropy is given by the exponential growth rate
of return times to mistake dynamical balls, that the minimal
return times to mistake dynamical balls grow linearly with respect
to its length and obtain some formula connecting the topological
pressure to weighted recurrence. Moreover, we also obtain a
generalization of an entropy formula due to Chazottes~\cite{ch} for
suspension semiflows. Since our main results require a mild specification
property we are able to give applications to the $\beta$-transformation
and the corresponding suspension semiflow.
Finally, we expect these results to have a wider range of applications and to
open the way to the study of other properties as the relation between recurrence to balls
(eventually for non-uniformly expanding maps) and pointwise
dimension, as well as the multifractal formalism for this notion of
mistake recurrence.

The remainder of this paper is organized as follows. In
Section 2 we state our main results and give some applications to
$\beta$-transformations and suspension semiflows. In Section 3 we
recall some definitions and present some preliminary results.
Finally, the proofs of the main results are given in Section 4.

\section{Statement of the main results}
\setcounter{section}{2} \setcounter{equation}{0}

In this section we give some definitions and set the context for
our main results. First we recall the definitions of mistake
function and mistake dynamical balls  which are due to Thompson,
Pfister and Sullivan \cite{ps,th}.

\begin{definition}
Given $\e_0>0$ the function $g:\mathbb{N}\times (0,\e_0]\rightarrow
\mathbb{N}$ is called a \emph{mistake function} if for all $\e\in
(0,\e_0]$ and all $n\in \mathbb{N}$, $g(n,\e)\leq g(n+1,\e)$ and
\[
\lim_{n\rightarrow\infty}\frac{g(n,\e)}{n}=0.
\]
By a slight abuse of notation we set $g(n,\e)=g(n,\e_0)$ for every $\e>\e_0$.
\end{definition}

For any subset of integers $\Lambda\subset [0,N]$, we will use the
family of distances in the metric space $X$ given by
$d_\Lambda(x,y)=\max \{d(f^ix,f^iy):i\in\Lambda\}$ and consider
the balls $B_{\Lambda}(x,\e)=\{y\in X:d_\Lambda(x,y)<\e\}$.

\begin{definition}
Let $g$ be a mistake function, $\e>0$ and $n\ge 1$. The \emph{mistake dynamical ball} $B_n(g;x,\e)$
of radius $\e$ and length $n$ associated to $g$ is defined by
\begin{eqnarray*}
B_n(g;x,\e)&=&\{ y\in X\mid y\in B_{\Lambda}(x,\e)~\hbox{for
some}~\Lambda\in I(g;n,\e)\}\\
&=&\bigcup_{\Lambda\in I(g;n,\e)}B_{\Lambda}(x,\e)
\end{eqnarray*}
where $I(g;n,\e)=\{ \Lambda\subset [0,n-1]\cap\mathbb{N}\mid
\# \Lambda \geq n-g(n,\e)\}$.
\end{definition}

For every mistake function $g$, we associate the \emph{first
return time} $R_n(g;x,\e)$ to the mistake dynamical ball
$B_n(g;x,\e)$ by $R_n(g;x,\e)=\inf \{ k\geq 1: f^k(x)\in
B_n(g;x,\e)\}.$ Our first result reflects a stability of the
metric entropy even if a small amount of errors is commited when
compared with the original orbits.

\begin{maintheorem}\label{thm.A}
Let $(X,f)$ be a TDS and let $g$ be any mistake function.
For every $\mu\in \mathcal{E}(X,f)$ the limits
 \[
 \overline{h}_g(f,x)=\lim_{\e\rightarrow 0}\limsup_{n\rightarrow
 \infty}\frac{1}{n}\log R_n(g;x,\e)~~and~~\underline{h}_g(f,x)=\lim_{\e\rightarrow 0}\liminf_{n\rightarrow
 \infty}\frac{1}{n}\log R_n(g;x,\e)
\]
exist for $\mu$-almost every $x$ and coincide with the
measure-theoretic entropy $h_{\mu}(f).$
\end{maintheorem}


Let us comment on the assumption of ergodicity in the above
theorem. Given any $\mu\in\mathcal{M}(X,f)$ by ergodic decomposition
theorem we know that $\mu$ can be decomposed as a convex
combination of ergodic measures, i.e. $\mu=\int \mu_x
\mathrm{d}\mu(x)$. Moreover, since $h_{\mu}(f)=\int
h_{\mu_x}(f)\mathrm{d}\mu(x)$, then applying Theorem A to each
ergodic component $\mu_x$ and integrating with respect to $\mu$ we
obtain the following consequence.

\begin{maincorollary}
Let $\mu\in\mathcal{M}(X,f)$.
Then the limits $\overline{h}_g(f,x)$ and $\underline{h}_g(f,x)$
defined above do exist for $\mu$-almost every $x$ . Moreover, the
measure-theoretic entropy satisfies \[ h_{\mu}(f)=\int
\overline{h}_g(f,x) \mathrm{d}\mu(x)=\int \underline{h}_g(f,x)
\mathrm{d}\mu(x).\]
\end{maincorollary}

Given a continuous observable $\phi:X\rightarrow\mathbb{R}$, the
measure-theoretic pressure $P_{\mu}(f,\phi)=h_{\mu}(f)+\int \phi\;
\mathrm{d}\mu$ of the invariant measure $\mu$ with respect to $f$
and $\phi$ can also be written using weighted recurrence times.
We refer the reader to \cite{zc} for more details on the measure-theoretic pressure for
a large class of potentials.

\begin{maincorollary}
Let $(X,f)$ be a TDS, $\mu\in\mathcal{E}(X,f)$, $\phi:X\rightarrow\mathbb{R}$ be a continuous potential.
Then, for every mistake function $g$ it holds
$$
P_{\mu}(f,\phi)
    =\lim_{\e\rightarrow 0}\limsup_{n\rightarrow  \infty}\frac{1}{n}\log [e^{S_n\phi(B_n(g;x,\e))} R_n(g;x,\e)]
 $$
for $\mu$-almost every $x\in X$, where $S_n\phi(B_n(g;x,\e))=\sup \{ \sum_{i=0}^{n-1}\phi(f^iy):y\in B_n(g;x,\e)\}$.
 \end{maincorollary}

 \begin{proof}
Let $\mu\in\mathcal{E}(X,f)$ and $\phi:X\rightarrow\mathbb{R}$ be a
continuous function. Given any $\d>0$, by the uniform continuity
of $\phi$, there
 exists $\e_\d>0$ such that $|\phi(x)-\phi(y)|<\d$ whenever
 $d(x,y)<\e$ for every $0<\e<\e_\d$. For each $y\in B_n(g;x,\e)$,
 there exists $\Lambda\subset I(g;n,\e)$ so that $y\in B_\Lambda
 (x,\e)$, therefore
 \begin{eqnarray*}
\sum_{i=0}^{n-1}\phi(f^iy)&\leq&\sum_{i\in\Lambda}(\phi(f^ix)+\d)+\sum_{i\notin\Lambda}||\phi||_\infty\\
&\leq&\sum_{i=0}^{n-1}(\phi(f^ix)+\d)+C g(n,\e),
 \end{eqnarray*}
where  $C=2(||\phi||_\infty+\d)$. 
Similarly, we have $
\sum_{i=0}^{n-1}\phi(f^iy)\geq\sum_{i=0}^{n-1}(\phi(f^ix)-\d)-Cg(n,\e).
$
 Hence
 \begin{eqnarray} \label{ds21}
 \sum_{i=0}^{n-1}(\phi(f^ix)-\d)-C g(n,\e)\leq S_n\phi(B_n(g;x,\e))\leq
 \sum_{i=0}^{n-1}(\phi(f^ix)+\d)+C g(n,\e).
 \end{eqnarray}
 On the other hand, using Birkhoff's ergodic theorem and Theorem~\ref{thm.A},
 there exists a $\mu$-full measure set $\mathcal{R}$ such that
 \[
 \lim_{n\rightarrow\infty}\frac{1}{n}\sum_{i=0}^{n-1}\phi(f^ix)=\int
 \phi \, \mathrm{d}\mu~~\hbox{and}~~h_{\mu}(f)=\lim_{\e\rightarrow 0}\limsup_{n\rightarrow
 \infty}\frac{1}{n}\log R_n(g;x,\e)
 \]
for every $x\in\mathcal{R}$. Now, given $x\in \mathcal{R}$, by
(\ref{ds21}) and the definition of mistake function, we have
\[
\left|
\limsup_{n\rightarrow
 \infty}[\frac{1}{n}S_n\phi(B_n(g;x,\e))+\frac{1}{n}\log
 R_n(g;x,\e)]-P_{\mu}(f,\phi)
 \right|<2\d
\]
for every small $\e>0$. Since $\delta$ was taken arbitrary, the result follows immediately.
 \end{proof}

We turn our attention to minimal return times. Namely, given a mistake function $g$ we define the $nth~minimal~return~time$ $S_n(g;x,\e)$ to the mistake dynamical ball $B_n(g;x,\e)$ by
\[
S_n(g;x,\e)=\inf \left\{ k\geq 1: f^{-k}(B_n(g;x,\e))\cap
B_n(g;x,\e)\neq \emptyset\right\}.
\]

Now we give an alternative definition of $g$-almost specification
property of a TDS $(X,f)$, motivated by the results from
Thompson's definition of  almost specification property \cite{th}
and Pfister and Sullivan's definition of $g$-almost product
property \cite{ps}.

\begin{definition}Let $g$ be a mistake function.
A TDS $(X,f)$ satisfies the $g$-almost specification property if
there exists $\e>0$ and a positive integer $N(g,\e)$ such that for
any $x,y\in X$ and integers $n,m \geq N(g,\e)$ we have
$B_{m}(g;y,\e) \cap f^{-m}(B_{n}(g;x,\e)) \neq \emptyset$.
\end{definition}

The previous notion is weaker than the one introduced in \cite{th}
since it deals with the case that any two pieces of approximate
orbits are given to be approximated by a real orbit within the
same scale $\e$. In opposition to \cite{ps} the unboundedness of
the mistake function is not required. It is interesting to study
the class of mistake functions $g$ for which $g$-almost
specification still holds, question which we discuss partially in
Lemma~\ref{l.mistake} below.
In what follows we prove that under some mild assumptions minimal return times grow linearly.
More precisely,

\begin{maintheorem}\label{thm.B}
Let $g$ be any mistake function. If $(X,f)$ is a TDS with $g$-almost specification property and
$\mu\in \mathcal{E}(X,f)$ so that $h_{\mu}(f)>0$, then the limits
 \[
 \overline{S}(x)=\lim_{\e\rightarrow 0}\limsup_{n\rightarrow
 \infty}\frac{1}{n}S_n(g;x,\e)~~and~~\underline{S}(x)=\lim_{\e\rightarrow 0}\liminf_{n\rightarrow
 \infty}\frac{1}{n}S_n(g;x,\e)
\]
exists and are equal to one for $\mu$-almost every $x$.
\end{maintheorem}

Note that mistake dynamical balls coincide with the usual
dynamical balls in the case that $g\equiv 0$.  Therefore,  this
theorem improves the results by the second author in
\cite[Theorem~B]{var} since we require as an hypothesis a weaker
specification property. We provide an interesting example that
illustrates this fact.

\begin{example}\label{ex.beta}
Consider the piecewise expanding maps of the interval $[0,1)$
given by $T_\beta(x)=\beta x (\!\!\!\mod 1)$, where $\beta>1$ is
not integer. This family is known as \emph{beta transformations}
and it was introduced by R\'enyi in \cite{re}.
It was proved by Buzzi~\cite{Buz97} that for all but countable
many values of $\beta$ the transformation $T_\beta$ do not satisfy
the specification property. These do not satisfy the conditions of
 \cite[Theorem~B]{var}.
It follows from \cite{ps,th} that every $\beta$-map satisfies the almost specification property
for every unbounded mistake function $g$.
Then it follows from our results 
that given any unbounded mistake function $g$ and every invariant measure
$\mu$ with positive entropy one has
$$
\lim_{\e\to 0}\limsup_{n\to \infty}\frac{1}{n}S_n(g;x,\e)
    =\lim_{\e\to 0}\liminf_{n\to  \infty}\frac{1}{n}S_n(g;x,\e)
    =1
$$
for $\mu$-almost every $x$.
\end{example}

\begin{remark}
It is not hard to use the previous results to obtain formulas relating entropy and return
times to partition elements instead of dynamical balls. In fact, let $\mu$ be an
$f$-invariant ergodic measure. If
$\mathcal Q$ is a partition of $X$ and $g=g(n,\mathcal Q)$ is any
function such that $g(n,\mathcal Q)\leq g(n+1,\mathcal Q)$ and
$\lim_{n\to\infty} g(n,\mathcal Q)/n=0$, also denoted by
mistake function, we can consider the mistake partition elements
$$
\mathcal Q_g^{(n)}=\bigcup_{\Lambda\in I(g;n,\mathcal Q)}\mathcal
Q^{(n)}_{\Lambda}
$$
where $I(g;n,\mathcal Q)=\{ \Lambda\subset
[0,n-1]\cap\mathbb{N}\mid \# \Lambda \geq n-g(n,\mathcal Q)\}$ and
$\mathcal Q^{(n)}_{\Lambda}=\bigvee_{j\in
\Lambda}f^{-j}\mathcal{Q}$.
One may endow the space $X$ with the pseudo-distance $d(x,y)=e^{-N}$, where
$N=\inf\{k\ge 1:  f^j(y)\in \mathcal Q(f^j(x)) \text{ for every }
1\le j \le k \}$, in which case the mistake dynamical ball
$B_n(g;x,\e)$ coincides with the union of partition elements
$\mathcal Q_g^{(n)}(x)$. We derive from Theorems~\ref{thm.A} and \ref{thm.B} that
\begin{equation}\label{eq.ent.partitions}
h_{\mu}(f,\mathcal Q)=\lim_{n\to\infty} \frac1n\log R_n(g,x, \mathcal Q)
    \quad\text{ for $\mu$-a.e $x$}
\end{equation}
and, if $f$ satisfies the $g$-almost specification property and $\mu$ has positive entropy then
\begin{equation}\label{eq.linear.partitions}
\lim_{n\to\infty} \frac{S_n(g,x, \mathcal Q)}n=1
\end{equation}
where $R_n(g,x,\mathcal Q)$ and $S_n(g,x,\mathcal Q)$ denote,
respectively, the first and the minimal return times of the point
$x$ to the set $\mathcal Q_g^{(n)}(x)$.
\end{remark}

We now obtain that recurrence is strongly related with topological
pressure. More precisely, despite the fact that we deal with
recurrence to mistake dynamical balls the first claim of the next
result can be understood as an extension of \cite{msuz} .

\begin{maintheorem}\label{thm.C}
Let $(X,f)$ be a subshift of finite type, $\phi: X\to\mathbb R$
be  a H\"older continuous potential and $\mu=\mu_\phi$ be the unique
equilibrium state of $f$ and $\phi$. Let us denote by $\cQ$ the partition of
$X$ into initial cylinders of length one. Then for every mistake
function $g$, it holds
$$
\lim_{n\to\infty} \frac1n \log \left[\sum_{j=0}^{R_n(g;x,\cQ)} e^{S_n\phi (f^j(x))}\right]
    =h_\mu(f) + P_{\text{top}}(f,2\phi) -P_{\text{top}}(f,\phi),
$$
moreover, if  $h_{\mu}(f)>0$, then
$$
\lim_{n\to\infty} \frac1n \log \left[\sum_{j=0}^{S_n(g;x,\cQ)} e^{S_n\phi (f^j(x))}\right]
    =P_{\text{top}}(f,2\phi) -P_{\text{top}}(f,\phi),
$$
for $\mu$-a.e. $x$, where $P_{\text{top}}$ stands for the topological pressure of $f$
with respect to $\phi$.
 \end{maintheorem}



An important remark is that in \cite{mv} the authors established a
similar formula for expansive dynamical systems with the
specification property, which contains the case of subshifts of
finite type. Although we will not prove it here it seems possible that
the theorem above may admit such a generalization for expansive
dynamical systems with the $g$-almost specification property.

Our last result will concern return times for suspension semiflows $(f^t)_t$
over a base dynamical system $\si:\Sigma \to \Sigma$ with continuous height
function $\varphi:\Si \to \R^{+}$ on a compact metric space $\Sigma$.
More precisely, $(f^t)_t$ acts on the space $Y=\{(x,s)
\in \Si \times \R^{+} : 0 \leq s \leq \varphi(x)\},$ where
$(x,\varphi(x))$ and $(\si(x),0)$ are identified for every $x\in\Sigma$, as the ``vertical flow"
defined by $f^t(x,s)=(x,s+t)$. With the natural identification on $Y$ we can write
$$
f^t(x,s)=\left( \si^k(x),t+s- \sum_{i=0}^k\varphi(f^i(x)) \right)
$$
whenever $\sum_{i=0}^k\varphi(f^i(x))\leq t+s \leq
\sum_{i=0}^{k+1}\varphi(f^i(x))$. It is well known that if the
roof function is bounded from away from zero and infinity then
there is a natural identification between the space $\cM$ of
$(f^t)_t$-invariant probability measures and the space $\cM_\si$
of $\si$-invariant probability measures. Namely,
\begin{equation}\label{eq.bijection}
\begin{array}{cccc}
L: & \cM_\si &    \to  &  \cM\\
   &     \mu       & \mapsto & \overline\mu=\frac{(\mu \times m)|_Y}{(\mu \times m)(Y)}
\end{array}
\end{equation}
is a bijection, where $m$ is the Lebesgue measure on $\mathbb{R}$.
In particular many ergodic properties for suspension semiflows can
be reduced to the study of the Poincar\'e return map corresponding
to the section $\Si\times\{0\}$. The following is an extension of
the results by Chazottes in \cite{ch}. First recall that the
first return time for flows have the subtlety that the return time is considered after the escaping time (for results on return times for flow we can refer to the thesis of the first autor \cite{thesej}).
Indeed, given an open set $A\subset Y$ and $(x,t)\in A$ define
the \emph{escaping time} of a point
$$
e_A((x,t))=\inf\left\{s>0 : f^s(x,t) \notin A\right\},
$$
the escaping time of a set
$$
e(A)=\inf\left\{s>0 : f^s(A) \cap A=\emptyset\right\}
$$
and the \emph{minimal return time} $\tau_f(A)$ as
$$
\tau_f(A)=\inf\left\{s>e(A) : f^{-s}(A)\cap A\neq \emptyset \right\}.
$$
We shall consider mainly sets of the form $A=B_n(g;x,\e)\times
[t-\e,t+\e]$ with respect to a mistake function $g$.

\begin{maintheorem} \label{thm.flows}
Let $g_1$ and $g_2$ be any mistake functions on $\Sigma$ and let $\mu$ be an ergodic $f$-invariant
probability measure. Assume that $f$ satisfies the $g_2$-almost specification property. If $\mu$ is an
$f$-invariant, ergodic probability measure with positive entropy then for $\mu$-almost
every $x\in\Sigma$ and every $s\in\R$ such that $(x,s)\in Y$, we have
\begin{eqnarray*}
h_{\overline\mu}((f^t))
    &=&\lim_{\e\to 0}\limsup_{n\to\infty}\frac{\log R_n(g_1;x,\eps)}{\tau_f(B_n(g_2;x,\e)\times(s-\eps,s+\eps))}\\
    &=&\lim_{\e\to 0}\liminf_{n\to\infty}\frac{\log R_n(g_1;x,\eps)}{\tau_f(B_n(g_2;x,\e)\times(s-\eps,s+\eps))}.
\end{eqnarray*}
where $\overline \mu$ is the $(f^t)$-invariant measure given by
\eqref{eq.bijection} and $R_n(g_1;x,\eps)$ stands for the first
return time of the point $x$ to the set $B_n(g_1;x,\e)$ by the
base transformation $\sigma$.
\end{maintheorem}

Let us mention that an adapted version of the previous result also holds for non-ergodic measures.
In fact, if $\overline\mu=L(\mu)$ is any $(f^t)$-invariant probability measure and $\mu=\int \mu_x \;d\mu(x)$
is an ergodic decomposition for $\mu$ then it is clear that $\overline\mu=\int L(\mu_x) \;d\mu(x)$ is an
ergodic decomposition of $\overline\mu$. Hence, we can use Theorem~\ref{thm.flows} above in the
formula $h_{\overline\mu}((f^t))=\int h_{L(\mu_x)}((f^t))\;d\mu(x)$.

\begin{example} For every topological Axiom A flow, that is, suspension semiflows
over subshifts of finite type, it is not hard to check in the
proof of Theorem~\ref{thm.flows} that a particular easy
application of Theorem~\ref{thm.C} yields
$$
\lim_{n\to\infty}\frac{\log \sum_{j=0}^{R_n(g_1;x,\cQ)} e^{S_n\phi(f^j(x))} }{\tau_f(B_n(g_2;x,\e)\times(s-\eps,s+\eps))}
    = h_{\overline\mu}((f^t)) + \frac{c_{\phi,1}}{\int \phi \,d\mu}
$$
where $c_{\phi,1}=P_{\text{top}}(f,2\phi)-P_{\text{top}}(f,\phi)$
is the free energy defined in (\ref{freeenergy32}).
 However, since we found no
particularly simple expression for the last term in the right
hand-side.  We shall not prove nor use this fact along the paper.
\end{example}

We give now an example of application of Theorem~\ref{thm.flows} to the suspension flow of
$\beta$-transformations.

\begin{example} \label{ex.beta.flow}
Take the $\beta$-transformation given by $T_\beta(x)=\beta x (\!\!\!\mod 1)$ in the interval $[0,1)$
with $\beta>1$ not integer as discussed in Example~\ref{ex.beta} and consider any unbounded
mistake function $g$ and any mistake function $\tilde g$. Let $\mu$ be any ergodic $T_\beta$-invariant
probability measure with positive entropy, $(f_t)_t$ be the suspension semiflow by a continuous roof function $\varphi$ bounded away from zero. Then $\overline{\mu}=(\mu\times Leb)/\int \varphi\, d\mu$ is a $(f_t)_t$-invariant ergodic probability measure satisfying
\begin{eqnarray*}
h_{\overline\mu}((f^t))
    &=&\lim_{\e\to 0}\limsup_{n\to\infty}\frac{\log R_n(\tilde g;x,\eps)}{\tau_f(B_n(g;x,\e)\times(s-\eps,s+\eps))}\\
    &=&\lim_{\e\to 0}\liminf_{n\to\infty}\frac{\log R_n(\tilde g ;x,\eps)}{\tau_f(B_n( g;x,\e)\times(s-\eps,s+\eps))}.
\end{eqnarray*}
for $\mu$-almost every $x$.
\end{example}

Further applications of our results exploring the relation between pointwise dimension, Lyapunov exponents
and entropy of invariant measures as in \cite{var} seems feasible and so we believe in an affirmative
answer to the following question. \vspace{.5cm}

\noindent {\bf Question:} Can one compute the pointwise dimension of an invariant measure
using recurrence to mistake dynamical balls?

\section{Preliminaries}
\setcounter{section}{3} \setcounter{equation}{0} In this section,
we recall some preliminary results about   entropy, free energy,
mistake function and suspension semiflows.
\subsection{Entropy}

 In this subsection, we first recall an equivalent description of the
 measure-theoretic entropy. Namely, using Katok's entropy
 formula \cite{kat1} and Shannon-McMillan-Breiman's theorem, we have the following lemma.

 \begin{lemma}Let $\mathcal{Q}$ be a partition of $X$, $c\in (0,1)$ and $\mu\in\mathcal{E}(X,f)$.   Then
 \begin{eqnarray}
 h_{\mu}(f,\mathcal{Q})=\limsup_{n\rightarrow
 \infty}\frac{1}{n}\log N_{\mu}(n,\e,c)
\end{eqnarray}
where $N_{\mu}(n,\e,c)$ denotes the minimum number of
$n$-cylinders of the partition
$\mathcal{Q}^{(n)}=\bigvee_{i=0}^{n-1}f^{-i}\mathcal{Q}$ necessary
to cover a set of $\mu$-measure at least  $c$.
\end{lemma}

Now we turn our attention to the following covering lemma for
mistake dynamical balls associated with points with slow
recurrence to the boundary of a given partition.

\begin{lemma} \label{yl32} Let $\mathcal{Q}$ be a finite partition
of $X$ and consider $\e>0$ arbitrary small. Let $V_\e$ denote the
$\e$-neighborhood of the boundary $\partial\mathcal{Q}$. For any
$\a>0$, there exists $\gamma>0$(depending only on $\a$),  such
that for every $x\in X$ satisfying
$\sum_{j=0}^{n-1}\chi_{V_\e}(f^jx)<\gamma n$, the mistake
dynamical ball $B_n(g;x,\e)$ can be covered by $e^{\a n}$
cylinders of $\mathcal{Q}^{(n)}$ for sufficiently large $n$.
\end{lemma}

\begin{proof}
Fix an arbitrary $\a>0$. Since $B(z,\e)\subset \mathcal{Q}(z)$ for
each $z\notin V_\e$, the itinerary of any point $y\in B_n(g;x,\e)
$ for which $\sum_{j=0}^{n-1}\chi_{V_\e}(f^jx)<\gamma n$ will
differ from the one of $x$ by at most $[\gamma n]+g(n,\e)$ choices
of partition elements. Since there are at most
 $$
 \left(
\begin{array}{c}
n \\
 \gamma n + g(n,\e)
\end{array}
 \right)
 \; \;
( \# \mathcal{Q})^{\gamma n+g(n,\e)}
 $$
such choices and $\lim_{n\rightarrow\infty}\frac{g(n,\e)}{n}=0$,
the previous upper can be made smaller than $e^{\a n}$ provided
that $\gamma>0$ is small and $n$ is sufficiently large. This
completes the proof of the lemma.
\end{proof}

\subsection{Free energy}\label{s.free.energy}

In this subsection,  we will define and collect some important
characterizations for the free energy of subshifts of finite
type. Given an observable $\phi: X\to \mathbb R$ the \emph{free
energy} is defined as
\begin{equation}\label{freeenergy32}
c_{\phi,t} = \limsup_{n\to\infty} \frac1n \log \int e^{S_n t \phi}
\,d\mu.
\end{equation}
This functional is very used in the physics and large deviations
literature since its  Legendre transform is an upper bound for the
measure of deviation sets. Let us recall a very interesting
formula for the free energy of subshifts of finite type.

\begin{lemma}\cite[Lemma~2.1]{msuz} \label{lemma33}
Let $f:X\to X$ be a topological mixing subshift of finite type,
let $\phi: X\to\mathbb R$ be a H\"older continuous potential and
let $\mu=\mu_\phi$ be the unique equilibrium state for $f$ and
$\phi$. Then the limit
$$
\lim_{n\to\infty} \frac1n \log \int e^{S_n t \phi} \,d\mu
    =P_{\text{top}}(f,(t+1)\phi) - P_{\text{top}}(f,\phi)
$$
does exist and coincides with the free energy $c_{\phi,t}$. In
particular,
$c_{\phi,1}=P_{\text{top}}(f,2\phi)-P_{\text{top}}(f,\phi)$
\end{lemma}

In this context the free energy is continuous and differentiable and it plays a key role in the theory
of large deviations. Indeed, it follows from Ellis Large Deviation Theorem that for every H\"older continuous $\psi$
$$
\lim_{n\to\infty} \frac1n \log
    \mu_\phi\left( x\in X:  \left|\frac1n \sum_{j=0}^{n-1} \psi(f^j(x)) - \int \psi\; d\mu_\phi\right|>\delta \right)
    =-\hat I_\phi(\delta)
$$
where $\hat I_\phi$ denotes the Legendre transform of the free energy function
$$
C_\phi\colon t\mapsto \lim_{n\to\infty} \frac1n \log \int e^{S_n t\psi} \; d\mu_\phi,
$$
that is, $\hat I_\phi (t)=\sup_{s} \{ st-C_\phi(s) \}$. Moreover, $\hat I_\phi$ is differentiable and attains a minimum $\hat I_\phi(0)=0$.  As a consequence of the variational  relationship between
$C_\phi$ and $\hat I_\phi$ we have also $C_\phi(t)=\sup_s\{ st+\hat I_\phi(s)\}$.

\subsection{Mistake functions}

In this subsection we show that the almost specification property does not depend on the unbounded
mistake function that we consider.

\begin{lemma}\label{l.mistake}
Let $g$ be a mistake function and $\e>0$. If a TDS $(X,f)$
satisfies the $g$-almost specification property with scale $\e$
then $f$ satisfies the $\tilde g$-almost specification property
for every mistake function $\tilde g$ satisfying
\begin{equation}\label{eq.cond}
\liminf_{k\to\infty } \;[\tilde g(\e,k)- g(\e,k)] \geq 0.
\end{equation}
\end{lemma}

\begin{proof}
Let the mistake function $g$ and $\e$ be fixed, and consider an arbitrary mistake function $\tilde g$ satisfying~\eqref{eq.cond}.  Since $f$ satisfies the $g$-almost specification property then there exists a positive integer
$N(g,\e)$ such that
\begin{equation}\label{eq.aux}
B_{n_1}(g;x_1,\e) \cap f^{-n_1}(B_{n_2}(g;x_2,\e)) \neq \emptyset
\end{equation}
for every $n_1,n_2\ge N(g,\e)$ and $x_1,x_2\in X$. Our assumption
yields that
$$
N(\tilde g,\e)=\max\left\{
        N( g,\e) , \inf\{k\in\mathbb N : \tilde g(\e,\ell) \geq g(\e,\ell) \text{ for every } \ell \geq k\}
        \right\}
$$
is well defined and finite. Moreover, it is clear that
$B_{n_i}(g;x_i,\e) \subset B_{n_i}(\tilde g;x_i,\e)$ for every $x_i\in X$, every $n_i\ge N(\tilde g,\e)$ and all $i=1,2$. In particular we deduce that
$$
B_{n_1}(\tilde g;x_1,\e) \cap f^{-n_1}(B_{n_2}(\tilde g;x_2,\e))
    \supset B_{n_1}(g;x_1,\e) \cap f^{-n_1}(B_{n_2}(g;x_2,\e))
    \neq \emptyset
$$
for every $n_1,n_2 \geq N(\tilde g,\e)$, which concludes the proof of the lemma.
\end{proof}

\subsection{Suspension semiflows}

In this subsection we recall the Abramov formula, that relates the entropy of a suspension semiflow with the entropy of invariant measures for the global Poincar\'e first return transformation, which will be of particular use in the
proof of Theorem~\ref{thm.flows}.

\begin{lemma}\label{l.Abramov}
Let $\sigma:\Si\to\Si$ be a continuous transformation, $\mu$ be a
$\sigma$-invariant probability measure and $\varphi:\Si \to
\R^{+}$ be a strictly positive and $\mu$-integrable roof function.
Then the associated suspension semiflow $(f^t)_t$ satisfies
$$
h_{\overline\mu}((f^t)_t):= h_{\overline\mu}(f^1)
    =\frac{h_\mu(\sigma)}{\int \varphi \, d\mu},
$$
where $\overline \mu$ is the $(f^t)$-invariant measure given by \eqref{eq.bijection}.
\end{lemma}

\section{Proof of the main results}
\setcounter{section}{4} \setcounter{equation}{0}

In this section, we will show our main results that relate entropy
and topological pressure with the first and minimal return times
to mistake dynamical balls.

\subsection{Proof of Theorem A} \begin{proof}First we note that the limits
in the statement of Theorem A are indeed well defined almost
everywhere. Given $n\geq 1, \e>0$ and $x\in X$, we claim that
\begin{eqnarray} \label{ds41}
R_n(g;x,\e)\geq R_{n-1}(g;f(x),\e).
\end{eqnarray} Indeed, $f^{R_n(g;x,\e)}(x)\in
B_n(g;x,\e)$ implies that $$f^{R_n(g;x,\e)}(f(x))\in
f(B_n(g;x,\e))\subset B_{n-1}(g;f(x),\e),$$ which immediately
implies the claim (\ref{ds41}). Define
\[
\overline{h}_g(x,\e)=\limsup_{n\rightarrow
 \infty}\frac{1}{n}\log R_n(g;x,\e)
\quad \text{ and }\quad
 \underline{h}_g(x,\e)=\liminf_{n\rightarrow
 \infty}\frac{1}{n}\log R_n(g;x,\e).
\]
It follows from (\ref{ds41}) that $\overline{h}_g(x,\e)\geq
\overline{h}_g(f(x),\e)$ and $\underline{h}_g(x,\e)\geq
\underline{h}_g(f(x),\e)$. Since $\mu\in \mathcal{E}(X,f)$, these
functions are almost everywhere constant and their value will be
denoted by $\overline{h}_g(\e)$ and $\underline{h}_g(\e)$
respectively. Put
\[
\overline{h}_g(f)=\lim_{\e\rightarrow
0}\overline{h}_g(\e)~~\hbox{and}~~\underline{h}_g(f)=\lim_{\e\rightarrow
0}\underline{h}_g(\e),
\]
such limits do exists by monotonicity of the functions
$\overline{h}_g(\e)$ and $\underline{h}_g(\e)$. Hence, to prove
the theorem, it suffices to prove the following inequalities
\begin{eqnarray} \label{ds42}
\overline{h}_g(f)\leq h_{\mu}(f)\leq \underline{h}_g(f).
\end{eqnarray}

 It is easy to prove the left hand side inequality in
 (\ref{ds42}). Since $B_n(x,\e)\subset B_n(g;x,\e)$ implies that
 $R_n(x,\e)\geq R_n(g;x,\e)$. Then using the previous results
in \cite{dw,var} we get
\[
\overline{h}_g(f)=\lim_{\e\rightarrow 0}\limsup_{n\rightarrow
 \infty}\frac{1}{n}\log R_n(g;x,\e)\leq \lim_{\e\rightarrow 0}\limsup_{n\rightarrow
 \infty}\frac{1}{n}\log R_n(x,\e)=h_{\mu}(f)
\]
for $\mu$-almost every $x$.

We are left to prove the second inequality in (\ref{ds42}).
Despite the fact that some admissible mistakes can occur their effect
is neglectable from the combinatorial point of view. We follow the strategy in \cite{var}
that we include here for completeness. Assume, by
contradiction that $h_{\mu}(f)> \underline{h}_g(f)$. We pick a
finite partition $\mathcal{Q}$ of $X$ such that
$\mu(\partial\mathcal{Q})=0$ and $h_{\mu}(f)\geq
h_{\mu}(f,\mathcal{Q})> b>a>\underline{h}_g(f)$. Fix $0<\gamma
<(b-a)/6$ small such that Lemma \ref{yl32} holds for $\a=(b-a)/2$.
Pick also a sufficiently small $\e>0$ so that the
$\e$-neighborhood $V_\e$ of the boundary $\partial\mathcal{Q}$ then
$\mu(V_\e)<\gamma/2$. By Birkhoff's ergodic theorem, there exists
$N_0>1$ large such that the following set
\[
A=\left\{ x\in X: \sum_{j=0}^{n-1}\chi_{V_\e}(f^jx)<\gamma n,~\forall
n\geq N_0\right\}
\]
has $\mu$-measure larger than $1-\gamma$. By Lemma \ref{yl32} each
mistake dynamical ball $B_l(g;z,\e)$ of length $l\geq N_0$
centered at any point $z\in A$ can be covered by $e^{\a l}$
cylinders of $\mathcal{Q}^{(l)}$. Furthermore, provided that
$N_1\geq N_0$ is large enough, the measure of the set
\[
B=\{ x\in X: \exists N_0\leq n\leq N_1~\hbox{s.t.}~R_n(g;x,\e)\leq
e^{\a n} \}
\]
is also larger than $1-\gamma$. For notational simplicity we shall
omit the dependence of the sets $A$ and $B$ on the integers $N_0$
and $N_1$. Using Birkhoff's ergodic theorem again, there exists
$N_2>1$ large such that the set
\[
\Gamma=\left\{ x\in X:\sum_{j=0}^{k-1}\chi_{A\cap
B}(f^jx)>(1-3\gamma)k,~\forall k\geq N_2\right\}
\]
has $\mu$-measure at least $1/2$. We claim that there exists a positive
constant $C$ so that $\Gamma$ is covered by $Ce^{bk}$ cylinders
of $\mathcal{Q}^{(k)}$, for every large $k$. This will lead to the contradiction
\[
h_{\mu}(f,\mathcal{Q})=\lim_{n\rightarrow\infty}\frac{1}{n}\log
N(k,\mathcal{Q},1/2)<b,
\]
and so proving the theorem.

In the following, we prove the previous claim. Fix $x\in \Gamma$ and
$k\gg N_2$. We proceed to divide the set $\{0,1,\cdots, k\}$ into
blocks according to the recurrence properties of the orbit of $x$. If
$x\notin A\cap B$ then we consider the block $[0]$. Otherwise, we
take the first integer $N_0\leq m\leq N_1$ such that
$R_m(g;x,\e)\leq e^{am}$ and consider the block $[0,1,\cdots,
m-1]$. We proceed recursively and, if $\{1,2,\cdots, k'\}$ (where $k'<k$)
is partitioned into blocks then the next block is $[k'+1]$ if
$f^{k'+1}(x)\notin A\cap B$ and  it will be
$[k'+1,k'+2,\cdots,k'+m']$ if $f^{k'+1}(x)\in A\cap B$ and $m'$ is
the first integer in $[N_0,N_1]$ such that
$R_{m'}(g;f^{k'+1}(x),\e)\leq e^{am'}$. This process will finish
after a finite number of steps and partitions $\{0,1,\cdots, k\}$
according to the recurrence properties of the iterates of $x$,
except possibly the last block which has size at most $N_1$. We
write the list of sequence of block lengths determined above as
$\iota(x)=[m_1,m_2,\cdots, m_{i(x)}]$. By construction there are
at most $3\gamma k$ blocks of size one. This enable us to give an
upper bound on the number of $k$-cylinders $\mathcal{Q}^{(k)}$
necessary to cover $\Gamma$. First, note that since each $m_i$ is
either one or larger than $N_0$, there are at most $k/N_0$ blocks
of size larger than $N_0$. Hence, there are at most
$$
\sum_{j\leq 3\gamma k}
 \left(
\begin{array}{c}
k/N_0+3\gamma k \\
j
\end{array}
 \right)
 \; \;
\leq 3\gamma k
\left(
\begin{array}{c}
k/N_0+3\gamma k \\
3\gamma k
\end{array}
 \right)
$$
possibilities
to arrange the blocks of size one. Now, we give an estimate on the
number of possible combinatorics for every prefixed configuration
$\iota=[m_1,m_2,\cdots, m_{l}]$ satisfying $\sum m_j=k$ and
$\# \{j:m_j=1\}<3\gamma k $. This will be done fixing elements
from the right to the left. Define $M_j=\sum_{i\leq j}m_i$. If
$x\in \Gamma$ is such that $\iota(x)=l$, there are at most $\#
\mathcal{Q}$ possibilities to choose a symbol for each block of
size one. Moreover, if $1\leq s \leq l$ is the first integer such
that $\sum_{i=s}^{l}m_i<N_1+e^{aN_1}$ then there are at most
$(\# \mathcal{Q})^{N_1+e^{aN_1}}$ possibilities for choices of
$(m_s+m_{s+1}+\cdots +m_l)$-cylinders with combinatorics
$[m_s,m_{s+1},\cdots ,m_l]$. Recall that
$R_{m_{s-1}}(g;f^{M_{s-2}}(x),\e)\leq e^{a m_{s-1}}\leq e^{aN_1}$
and, by Lemma \ref{yl32}, the mistake dynamical ball
$B_{m_{s-1}}(g;f^{M_{s-2}}(x),\e)$ can be covered by at most
$e^{\a m_{s-1}}$ cylinders in $\mathcal{Q}^{(m_{s-1})}$. Hence,
the possible itineraries for the $m_{s-1}$ iterates
$\{f^{M_{s-2}}(x),\cdots ,f^{M_{s-1}}(x)\}$ may be chosen among
$e^{\a m_{s-1}}$ options corresponding to each of the $e^{a
m_{s-1}}$ previously possibly distinct and fixed blocks of size
$m_{s-1}$ in $[m_s,\cdots, m_l]$. This shows that there are at
most $e^{(a+\a)m_{s-1}}$ possible itineraries for the $m_{s-1}$
iterations of $f^{M_{s-2}}(x)$. Proceeding recursively for
$m_{s-2},\cdots, m_2,m_1$ we conclude, after some steps, that
there exists $C>0$(depending only on $N_1$) such that if $\gamma$
was chosen small then $\Gamma$ can be covered by
$$
3\gamma k
\left(
\begin{array}{c}
k/N_0+3\gamma k \\
3\gamma k
\end{array}
 \right)
(\#
\mathcal{Q})^{N_1+e^{aN_1}+3\gamma k}e^{(a+\a)k}\leq Ce^{bk}
$$
cylinders in $\mathcal{Q}^{(k)}$. This proves the claim and
finishes the proof of the theorem.\end{proof}

\subsection{Proof of Theorem B}\begin{proof} Given a mistake function $g$.  First, we note that the
 $g$-almost specification property guarantees that for every small
$\e>0$  there exists an integer $N(g,\e)$ such that for each $x\in
X$ and $n\geq N(g,\e)$ we have $B_{n}(g;x,\e) \cap
f^{-n}(B_{n}(g;x,\e)) \neq \emptyset$. Therefore, we have
\[
\limsup_{n\rightarrow\infty}\frac{1}{n}S_n(g;x,\e)\leq 1
\]
for every small $\e>0$. In particular $\overline{S}(x)\leq 1$ for each $x\in X$.

Next, we prove that $\underline{S}(x)\geq 1$ for $\mu$-almost
every $x$. We claim that, for any $0<\eta <1$, there exists a
measurable set $E_\eta$ with $\mu(E_\eta)>1-\eta$ and
\[
\mu(\{x\in E_\eta: S_n(g;x,\e)\leq \eta n\} )
\]
is summable for every small $\e$. Using Borel-Cantelli lemma it
will follow that $\mu$-almost every $x\in E_\eta$ satisfies $ S_n(g;x,\e)>
\eta n$ for all but finitely many values of $n$ and every small
$\e$. Then the desired result will follow from the arbitrariness
of $\eta$.

 We are only left to  prove the claim above. Let $\eta \in
 (0,1)$ and fix a small
 $0<\a<\frac{1}{3}(1-\eta)h_{\mu}(f)$. Consider a finite partition
 $\mathcal{Q}$ of $X$ with $\mu(\partial\mathcal{Q})=0$ and
 $3\a<(1-\eta)h$, where $h=h_{\mu}(f,\mathcal{Q})>0$. If $\e_0$ is
 small enough then $\mu(V_\e)<\gamma/2$ for every $0<\e<\e_0$,
 where $\gamma=\gamma(\a)>0$ is given as in Lemma~\ref{yl32}.  Using
 Birkhoff's ergodic theorem together with the Shannon-McMillan-Breiman's theorem
 we deduce that  for almost every $x$, there exists an integer $N(x)\geq 1$ so
 that for every $n\geq N(x)$
 \begin{eqnarray} \label{ds43}
 \sum_{j=0}^{n-1}\chi_{V_\e}(f^j(x))<\gamma n
 \quad
 \text{and}
 \quad
 e^{-(h+\a)n}\leq \mu(\mathcal{Q}^{(n)}(x))\leq
 e^{-(h-\a)n}
 \end{eqnarray}
 where $\mathcal{Q}^{(n)}(x)$ denotes the element of
 $\mathcal{Q}^{(n)}$ which contains $x$. By Lemma~\ref{yl32}, each
 mistake dynamical ball $B_n(g;x,\e)$ is covered by a collection
 $\mathcal{Q}^{(n)}(g,x,\e)$ of $e^{\a n}$ cylinders of the
 partition $\mathcal{Q}^{(n)}$. Pick $N\geq 1$ large such that the
 following set
 $$E_\eta=\{x\in X: x~\hbox{satisfying}~(\ref{ds43})~,\forall n\geq N\} $$
 has measure bigger than $1-\eta$. Since $\mathcal{Q}$ is finite,
 there exists $K>0$ such that
 \[
K^{-1}e^{-(h+\a)n}\leq \mu(\mathcal{Q}^{(n)}(x))\leq K
e^{-(h-\a)n}
 \]
 for every $x\in E_\eta$ and every $n\geq 1$. We consider now the level sets
 $E_{\eta}(n,k)=\{x\in E_\eta: S_n(g;x,\e)=k\}$ and observe that
 $B_n(g;x,\e)\subset \bigcup\limits_{Q_n \in \mathcal{Q}^{(n)}(g,x,\e)}Q_n$. Thus,
 if $x\in E_{\eta}(n,k)$, then the mistake
dynamical ball $B_n(g;x,\e)$ is contained in the sub-collection
 of cylinders $Q_n\in \mathcal{Q}^{(n)}(g,x,\e)$ whose iteration by
 $f^k$ intersects any of the $n$-cylinders of
 $\mathcal{Q}^{(n)}(x,\e)$. Any such cylinders $Q_n$ are naturally determined
 by their first $k$ symbols and by the at most $e^{\a n}$ possible
 strings following them. So, the number of those cylinders is
 bounded by $e^{\a n}$ times the number of cylinders in
 $\mathcal{Q}^{(k)}$ that intersect $E_\eta$, that is, $e^{\a n}K
 e^{(h+\a)k}$. Hence, if $n\geq N$, we have
 \[
 \mu(\{x\in E_\eta: S_n(g;x,\e)<\eta n\})\leq \sum_{k=1}^{\eta n}
 \sum_{Q_n\cap E_\eta(n,k)\neq \emptyset}\mu(Q_n)
    \leq K\eta n \, e^{-(h-2\a)n}e^{(h+\a)\eta n}
 \]
 which is summable because $(h-2\a)-(h+\a)\eta >(1-\eta)h-3\a>0$.
 This proves our claim and finishes the proof of theorem B.\end{proof}

\subsection{Proof of Theorem C}

Here we prove the two inequalities of Theorem~C independently. The first one, inspired by Lemma~2.1 in \cite{msuz}, does not depend on the TDS.

\begin{lemma}\label{l.upper1}
Given $\mu\in\cE(X,f)$, a mistake function $g$ and a partition $\cQ$ then
$$
\limsup_{n\to\infty} \frac1n \log \left[\sum_{j=0}^{R_n(g;x,\cQ)} e^{S_n\phi (f^j(x))}\right]
    \leq h_\mu(f,\cQ) + c_{\phi,1},\ \mu-a.e.\ x.
$$
Moreover, if $\mu\in\cE(X,f)$ with $h_{\mu}(f)>0$, then
$$
\limsup_{n\to\infty} \frac1n \log \left[\sum_{j=0}^{S_n(g;x,\cQ)} e^{S_n\phi (f^j(x))}\right]
    \leq  c_{\phi,1},\ \mu-a.e.\ x.
$$
\end{lemma}

\begin{proof}
The arguments used are modifications of the arguments in
\cite[Lemma~2.1]{msuz} together with equations
\eqref{eq.ent.partitions} and \eqref{eq.linear.partitions}. We
include the proof here for the reader's convenience. Let $\de>0$
be small and fixed, and let $N\ge 1$ be large enough so that the
sets $X^1_\delta=\{x\in X : \log R_n(g;x,\cQ) \leq (h_\mu(f)+\de)
n, \forall n\geq N\} $ and $X^2_\delta=\{x\in X : S_n(g;x,\cQ)
\leq (1+\de) n, \forall n\geq N\} $ have measure at least $1-\de$. Let
$a_1(n)=e^{(h_\mu(f)+\de)n}$ and $a_2(n)=(1+\de)n$, then one can
use the Tchebychev's inequality and the invariance of the measure
$\mu$ to deduce that the measure of the set
$$
A^i_n(\de)
    =\left\{
    x\in X^i_\de : \sum_{j=0}^{a_i(n)} e^{S_n\phi (f^j(x))} > a_i(n)e^{(c_{\phi,1}+\de)n}
     \right\}~~(i=1,2)
$$
is bounded from above by
$$
\mu(A^i_n(\de))
    \leq \frac{1}{a_i(n)}e^{-(c_{\phi,1}+\de)n} \, \int \sum_{j=0}^{a_i(n)} e^{S_n\phi (f^j(x))} \,d\mu
    \leq e^{-\de n}  \; e^{-c_{\phi,1} n} \, \int e^{S_n\phi} \,d\mu,
$$
which is summable by Lemma \ref{lemma33}.  Using Borel-Cantelli
lemma it follows that $\mu$-a.e. $x\in X^i_\de (i=1,2)$ satisfies
\begin{eqnarray*}
\limsup_{n\to\infty} \frac1n \log \left(\sum_{j=0}^{R_n(g;x,\cQ)} e^{S_n\phi (f^j(x))}\right)
    &\leq& \limsup_{n\to\infty} \frac1n \log \left(\sum_{j=0}^{a_1(n)} e^{S_n\phi (f^j(x))}\right)\\
     &\leq& h_\mu(f) + c_{\phi,1}+2\de
\end{eqnarray*}
and
$$
\limsup_{n\to\infty} \frac1n \log \left(\sum_{j=0}^{S_n(g;x,\cQ)} e^{S_n\phi (f^j(x))}\right)
    \leq \limsup_{n\to\infty} \frac1n \log \left(\sum_{j=0}^{a_2(n)} e^{S_n\phi (f^j(x))}\right)
     \leq c_{\phi,1}+\de.
$$
The result follows from the arbitrariness of $\de$.
\end{proof}

In the remaining of the proof assume that $f:X\to X$ is a unilateral subshift of finite type and $\mathcal Q$ is a finite
Markov partition.
The second lemma uses more specific characterization of free
energy for equilibrium states associated to subshifts of finite
type as described before in subsection~\ref{s.free.energy}.

\begin{lemma}
Assume that $f:X\to X$ is a subshift of finite type,
$\mu=\mu_\phi$ is the unique equilibrium state with respect to the
H\"older continuous potential $\phi: X\to \mathbb R$ and $\cQ$  is
the partition of $X$ into initial cylinders of length one. Then
for every mistake function $g$,  it follows that
$$
\liminf_{n\to\infty} \frac1n \log \left[\sum_{j=0}^{R_n(g;x,\cQ)} e^{S_n\phi (f^j(x))}\right]
    \geq h_\mu(f) + c_{\phi,1}
$$
moreover, if  $h_{\mu}(f)>0$, then
$$
\liminf_{n\to\infty} \frac1n \log \left[\sum_{j=0}^{S_n(g;x,\cQ)} e^{S_n\phi (f^j(x))}\right]
    \geq  c_{\phi,1}
$$
for $\mu$-almost every $x$.
\end{lemma}

\begin{proof}
We assume without loss of generality  that $P_{\text{top}}(\phi)=0$
and $\phi<0$, otherwise just take $\psi=\phi-P_{\text{top}}(\phi)$
which has zero topological pressure and, since equilibrium states associated to
subshifts of finite type have positive entropy then there exists some positive integer
$k$ such that $S_k\psi<0$. Recall that  the unique equilibrium state $\mu$
is always ergodic.

If $\phi$ is cohomologous to a constant, that is,
$\phi=\varphi\circ f -\varphi +c$ for some H\"older continuous $\varphi$ and $c\in\mathbb R$
then $0=P(\phi)=P(\varphi\circ f -\varphi +c)=P(c)=h_\text{top}(f)+c$ shows that $c=-h_\text{top}(f)$.
Analogously $P(2\phi)=h_\text{top}(f)+2c=-h_\text{top}(f)$. Hence,
using \eqref{eq.ent.partitions} and that $\cQ$ is a generator it follows that
$$
\lim_{n\to\infty} \frac1n \log \left[\sum_{j=0}^{R_n(g;x,\cQ)}
e^{S_n\phi (f^j(x))}\right]
    = 0
    = h_{\mu}(f) +
    P_{\text{top}}(2\phi)-P_{\text{top}}(\phi).
$$

So, it remains to consider only the case that $\phi$ is not cohomologous to a constant.
Pick $\delta>0$ small. Since
$\cQ$ is a generating partition for $f$ then one can pick $N\ge 1$
large enough so that the sets $X^1_\delta=\{x\in X : \log
R_n(g;x,\cQ) \geq (h_\mu(f)-\de) n, \forall n\geq N\} $ and
$X^2_\delta=\{x\in X : S_n(g;x,\cQ) \geq (1-\de) n, \forall n\geq
N\} $ have measure larger than $1-\de$. Consider also
$a_1(n)=e^{(h_\mu(f)-\de)n}$ and $a_2(n)=(1-\de)n$.
Since $P_{\text{top}}(\phi)=0$ and $\mu$ is an equilibrium state then $h_\mu(f)=-\int \phi \,d\mu$.
In consequence, 
\begin{align*}
\sum_{j=0}^{a_i(n)} e^{S_n\phi (f^j(x))}
    & = e^{-h_{\mu}(f)\, n} \; \sum_{j=0}^{a_i(n)} e^{S_n(\phi-\int \phi \,d\mu) (f^j(x))}\\
    & \geq e^{-(h_{\mu}(f)-\de)\, n} \; \# \left\{ 0\le j \le a_i(n) : \frac1n \left( S_n\phi (f^j(x)) -n \int \phi\, d\mu \right)
     >\de \right\} \\
     & \geq a_i(n) e^{-(h_{\mu}(f)-\de)\, n} \;
        \frac{\# \left\{ 0\le j \le a_i(n) : f^j(x)\in B_n(\de)\right\}}{a_i(n)}
\end{align*}
which, by ergodicity of $\mu$, is larger than $\frac12 a_i(n) e^{-(h_{\mu}(f)-\de)\, n} \mu(B_n(\de))$ provided that $n$
is large enough and where
$$
B_n(\de)
    =\left\{x\in X : \exists y \in \cQ^{(n)}(x) \text{ s.t. }  \frac1n \left( S_n\phi (y) -n\int \phi\, d\mu \right)
    >\de \right\}.
$$
The previous reasoning shows that for every $i=1,2$ and every $x\in X^i_\de$ will satisfy
\begin{align}
\liminf_{n\to\infty} \frac1n \log \left[\sum_{j=0}^{a_i(n)} e^{S_n\phi (f^j(x))}\right]
    & \geq \liminf_{n\to\infty} \frac1n \log [a_i(n) e^{-h_\mu(f) n}] \label{e.parte1} \\
    & + \left[ \de + \liminf_{n\to\infty} \frac1n \log \mu(B_\de(n)) \right]. \label{e.parte2}
\end{align}
It follows from the large deviations argument from \cite[page~10]{msuz} that the last summand is at
least $\sup_\de\{\de-\hat I(\de)\}$, where $\hat I(\de)=I(-h_\mu(f)+\de)$ and  $I$ denotes the Legendre transform
of the free energy at $\de$. Since the Legendre transform of $I$ is the free energy $c_{\phi,\mu}$ one gets that \eqref{e.parte2} is bounded from below by $c_{\phi,\mu}(1)+h_\mu(f)$.
Then, using that \eqref{e.parte1} is equal to zero when $i=1$ and is equal to $-h_\mu(f)$ when $i=2$, the lemma follows from the choice of the sets $X^i_\de$ and the arbitrariness of $\delta$.
\end{proof}

\subsection{Proof of Theorem D}

The present proof is inspired by some ideas of \cite{acs,ch}.
Given $\eta>0$ and $\delta>0$ small it follows as a simple consequence of
Birkhoff ergodic theorem and Theorems~\ref{thm.A} and~\ref{thm.B}
that there exists $\tilde\Sigma\subset\Sigma$ such that
$\mu(\tilde\Sigma)>1-\eta$ and  that the convergence is uniform in
$\tilde\Sigma$, that is, there exist uniform constants $\e_0>0$
and $N=N(\e_0)\in\mathbb N$ such that
\begin{equation}\label{eqbirkdelta}
\left|\frac{1}{n}\sum_{k=0}^{n-1}\varphi(f^kx)-\int\varphi d\mu\right| <\delta,
\end{equation}
\begin{equation}\label{eqrndelta}
\left|\frac{1}{n}\log R_n(g_1;x,\e)-h_\mu(f)\right|<\delta
\end{equation}
and
\begin{equation}\label{eqsndelta}
\left|\frac{1}{n}S_n(g_2;x,\e)-1\right|<\delta
\end{equation}
for every $0<\e<\e_0$, every $n\ge N$ and every $x\in \tilde\Sigma$.
Throughout the continuation of the proof, we assume that $\e$ is small and $n$ is large enough.
Given $(x,s)\in Y$ we also remark that
\begin{eqnarray*}
\tau_f(B_n(g_2;x,\eps)\times(s-\eps,s+\eps))&=&\tau_f(B_n(g_2;x,\eps)\times\{s\})-2\eps\\
&=&\underset{y\in B_n(g_2;x,\eps)}\inf\sum_{k=0}^{\tau_f(y,B_n(g_2;x,\eps))-1}\varphi(f^ky)-2\eps.
\end{eqnarray*}
where $\tau_f(y,B_n(g_2;x,\eps)):=\inf\{k\geq 1:f^ky\in B_n(g_2;x,\eps)\}$. By definition of a mistake
dynamical ball, for all $y\in B_n(g_2;x,\eps)$, it exists
$\Lambda_n(y)\subset\{0,\dots,n-1\}$ satisfying $\# \Lambda_n(y)\geq n-g_2(n,\eps)$ and
 such that $f^ky\in B(f^kx,\eps)$ for all $k\in\Lambda_n(y)$.

Now notice that if $y\in B_n(g_2;x,\eps)$ and $k\in\Lambda_n(y)$ then
\[
|\varphi(f^ky)-\varphi(f^kx)|\leq \alpha(\eps)
\]
where $\alpha(\eps)=\underset{z\in\Sigma}{\sup}\{|\varphi(z_1)-\varphi(z_2)|:z_1,z_2\in B(z,\eps)\}$ tends to
zero as $\eps$ tends to zero by uniform continuity of $\varphi$ in $\Sigma$.
Using also that  $|\varphi(f^ky)-\varphi(f^kx)|\leq 2||\varphi||_\infty$ for every
$k\in\cap\{0,\dots,n-1\}\setminus  \Lambda_n(y)$ one immediately gets
\begin{eqnarray}\label{eqalphalamg2}
\left|\sum_{k=0}^{n-1}\varphi(f^ky)-\varphi(f^kx)\right|&\leq&
\#\Lambda_n(y) \;\alpha(\eps)+ 2\|\varphi\|_\infty \; g_2(n,\eps)
\end{eqnarray}
for all $n>N$ and $y\in B_n(g_2;x,\eps)$.
Hence, given $n$ such that $\lfloor n(1-\delta)\rfloor>N$, equations
\eqref{eqsndelta}, \eqref{eqbirkdelta} and \eqref{eqalphalamg2} yield that
\begin{align*}
\underset{y\in B_n(g_2;x,\eps)}\inf\sum_{k=0}^{\tau_f(y,B_n(g_2;x,\eps))-1}\varphi(f^ky)&\geq \underset{y\in B_n(g_2;x,\eps)}\inf\sum_{k=0}^{S_n(g_2;x,\eps)-1}\varphi(f^ky)\\
&\geq \underset{y\in B_n(g_2;x,\eps)}\inf\sum_{k=0}^{\lfloor n(1-\delta)\rfloor-1}\varphi(f^ky),
\end{align*}
which, by construction, is bounded from below by
\begin{align*}
\sum_{k=0}^{\lfloor n(1-\delta)\rfloor-1} & \varphi(f^k x) - \alpha(\e) \lfloor (1-\delta n) \rfloor
    -  2\|\varphi\|_\infty\; g_2(\lfloor n(1-\delta)\rfloor,\eps) \\
& \geq \lfloor n(1-\delta)\rfloor\left(\int\varphi d\mu-\delta-\alpha(\eps)-2\|\varphi\|_\infty\frac{g_2(\lfloor n(1-\delta)\rfloor,\eps)}{\lfloor n(1-\delta)\rfloor}\right).
\end{align*}
On the other direction, consider a point $y_1\in B_n(g_2;x,\eps)$ for which the equality
$\tau_f(y_1,B_n(g_2;x,\eps))=S_n(g_2;x,\eps)$ holds. Then, a reasoning analogous to the
previous one is enough to show that
\begin{align*}
\underset{y\in B_n(g_2;x,\eps)}\inf & \sum_{k=0}^{\tau_f(y,B_n(g_2;x,\eps))-1}\varphi(f^ky) \\
    & \leq \lceil n(1+\delta)\rceil\left(\int\varphi d\mu+\delta+\alpha(\eps)+2||\varphi||_\infty\frac{g_2(\lceil n(1-\delta)\rceil,\eps)}{\lceil n(1+\delta)\rceil}\right).
\end{align*}
Finally, these lower and upper estimates together with \eqref{eqrndelta} give that the term
$
\frac{\log R_n(g_1;x,\e)}{\tau_f(B_n(g_2;x,\eps)\times(s-\eps,s+\eps))}
$
is bounded from below by
\begin{equation}\label{eqlogrnsurtausup}
\frac{n(1-\delta)h_\mu(f)}{\lceil n(1+\delta)\rceil\left(\int\varphi d\mu+\delta+\alpha(\eps)+2||\varphi||_\infty\frac{g_2(\lceil n(1+\delta)\rceil,\eps)}{\lceil n(1+\delta)\rceil}\right)-2\eps}
\end{equation}
and bounded from above by
\begin{equation}\label{eqlogrnsurtauinf}
\frac{n(1+\delta)h_\mu(f)}{\lfloor n(1-\delta)\rfloor\left(\int\varphi d\mu-\delta-\alpha(\eps)-2||\varphi||_\infty\frac{g_2(\lfloor n(1-\delta)\rfloor,\eps)}{\lfloor n(1-\delta)\rfloor}\right)-2\eps}.
\end{equation}
for every large $n$.
The theorem is now obtained for $\mu$-almost every $x$ using Abramov formula (see Lemma~\ref{l.Abramov}), the
arbitrariness of $\de$ and $\eta$, taking the limit superior (respectively the limit inferior) when $n$ tends to infinity in \eqref{eqlogrnsurtausup} and \eqref{eqlogrnsurtauinf} and then the limit when $\eps\rightarrow0$.

\end{document}